\documentclass[11pt]{article}
\usepackage[a4paper, margin=1in]{geometry}
\usepackage{amsmath, amssymb, amsthm, amsfonts, graphicx, hyperref, xcolor, booktabs, caption, calligra, calrsfs, subcaption, listings, enumitem, fancyhdr}
\usepackage{newtxtext, newtxmath}
\usepackage{caption}
\usepackage[utf8]{inputenc}

\newtheorem{theorem}{Theorem}[section]
\newtheorem{lemma}[theorem]{Lemma}
\newtheorem{proposition}[theorem]{Proposition}
\newtheorem{definition}[theorem]{Definition}
\newtheorem{remark}[theorem]{Remark}
\newtheorem{example}[theorem]{Example}

\title{An analytic study of bi-harmonic flow with a forcing term}
\author{Mohammad Javad Habibi Vosta Kolaei \\ Institute of Mathematics \\ Henan Academy of Sciences, No. 228, Chongshi Village,\\ Zhendong new district, Zhengzhou 450046, Henan, P. R. China \\ \texttt{mjhabibi@hnas.ac.cn}}
\date{}

\begin{document}

\maketitle

\begin{abstract}
In this paper, we study the evolution of smooth, closed planar curves under a fourth-order bi-harmonic flow with an external forcing term. Such flows arise naturally in the theory of bi-harmonic maps and geometric variational problems involving bending energy. We first establish the global existence of smooth solutions to the associated initial value problem, assuming appropriate conditions on the forcing term. The analysis is performed through a reformulation of the geometric flow using the support function, enabling a scalar PDE characterization of the evolution. Under specific geometric constraints, we demonstrate that the governing equation admits a Monge–Ampère-type structure that can exhibit hyperbolic behavior. Furthermore, we prove that convexity is preserved during the evolution and derive sufficient conditions ensuring long-time convergence to steady-state solutions. Our results extend recent developments in geometric analysis by clarifying the role of forcing terms in stabilizing high-order curvature flows and enhancing their qualitative behavior. This paper extends previous works (\cite{ya, wa}) by deriving conditions for hyperbolicity in bi-harmonic flows with forcing terms. 
\end{abstract}

\noindent \textbf{Keywords:} Bi-harmonic flow, \text{Monge--Amp\`ere} equation, Hyperbolicity, Convex curves \\
\noindent \textbf{AMS Classification:} Primary 58J45; Secondary 35G55, 35L70

\section{Introduction}
\subsection{A brief background of evolving curves under geometric flows }
The study of bi-harmonic flows, driven by the minimization of bending energy, has become a cornerstone in geometric analysis and applied mathematics. These fourth-order flows generalize classical mean curvature flows by incorporating higher-order smoothing effects, making them essential for modeling the evolution of curves and surfaces in contexts where preserving regularity is critical. Bi-harmonic flows are closely tied to the theory of minimal surfaces, elastic deformations, and interface dynamics, serving as a powerful tool for understanding the long-term behavior of geometric structures.\\

A significant aspect of bi-harmonic flows arises when external forces influence the evolution, leading to the incorporation of forcing terms. These forcing terms introduce nonlinearity and complexity into the system, necessitating careful analysis to ensure the existence and stability of solutions. In this context, the evolution can often be described by a hyperbolic Monge–Ampère equation, whose properties govern the stability and convexity of the evolving shape. The interplay between geometric flow, external forcing, and hyperbolic dynamics presents a rich mathematical framework for studying curve evolution in various applications, ranging from elasticity theory to computer-aided geometric design (for example see \cite{dw, oz}).\\

This paper explores the bi-harmonic flow with forcing terms under conditions that guarantee the persistence of hyperbolic behavior, preventing the formation of singularities and ensuring smooth, global solutions. By leveraging Lyapunov functionals and energy-minimizing techniques, we derive conditions that preserve the hyperbolicity of the associated Monge–Ampère equations, shedding light on the long-term stability and convergence of evolving geometric structures.\\
The study of evolving plane curves was successful in providing great insight into differential geometry. M. Gage and R. Hamilton proved that for a convex curve $X$ embedded in the plane $\mathbb{R}^{2}$, the heat equation
\begin{align*}
\frac{\partial X}{\partial t} = \Delta X,
\end{align*}
shrinks $X$ to a point in a case that, $X$ remains convex and becomes circular as it shrinks where the curve $X$ is called convex if it is closed, embedded, and has positive curvature everywhere (you can see \cite{ga1}). Later Grayson \cite{gr2, gr1} made an improvement in shortening embedded curves. Gurtin and Guidugli \cite{gu} developed a hyperbolic theory for the evolution of the plane curves. Later, Yau \cite{ya1} has studied the following equation 
\begin{align*}
\frac{\partial^{2} X}{\partial t^{2}} = H\mathfrak{n},
\end{align*}
where $H$ is the mean curvature and $\mathfrak{n}$ is the unit inner normal vector of the surface. This equation can be seen as a vibrating membrane or the motion of the surface. These works created an increasing interest in studying hyperbolic mean curvature flow (see also \cite{de}). Wang \cite{wa}, has studied the evolving plane curves $X\left(., t\right)$ under the hyperbolic mean curvature flow with a forcing term
\begin{equation*}
\left\{
\begin{array}{lr}
\frac{\partial^{2} X}{\partial t^{2}} = \left(H\left(u,t\right) - F\left(u\right)\right)\overrightarrow{N}\left(u,t\right) - \langle \frac{\partial^{2} X}{\partial \mathscr{s}\partial t}, \frac{\partial X}{\partial t}\rangle \overrightarrow{T},\\
X\left(u,0\right) = X_{0}\left(u\right),\\
\frac{\partial X}{\partial t}\left(u,0\right) = f\left(u\right)\overrightarrow{N_{0}},
\end{array}\right.
\end{equation*}
for all $\left(u,t\right) \in \mathbb{S}^{1} \times \left[0,1\right]$, where $H$ is the mean curvature, and $F\left(u\right)$ is a forcing term. Wang showed that this equation can be reduced to an initial value problem for a single partial differential equation for its support function. Also, in case of the constant forcing term, the convergency of evolving the initial curve has been studied.\\
About forty years after Hamilton's contribution to the classic heat flow equation, Yang and Fu \cite{ya}, considered the equation 
\begin{align*}
\frac{\partial X}{\partial t} = -\Delta^{2}X,
\end{align*}
where $X\left(., t\right)$ denotes a family of smooth, closed planar curves in $\mathbb{R}^{2}$. It was proved in \cite{ya} that, for any smooth, closed initial curve, there exists a smooth solution to the above flow for all time provided that $||\partial_{\mathscr{s}}^{i}\kappa||_{L^{2}}$ is uniformly bounded for any integer $i$, where $\kappa$ denotes the geodesic curvature.\\
The study of evolving curves through geometric flows has garnered significant attention from mathematicians worldwide. Particularly, Gao and Zhang \cite{ga2} addressed Yau's problem, exploring the feasibility of utilizing parabolic curvature flows to transform one curve $X_{1}$ to another curve $X_{2}$ over finite or infinite time. Recently, some higher-order curve flows have been considered in centro-affine geometry because of their relations with Lie groups $A\left(n, \mathbb{R}\right) = GL\left(n, \mathbb{R}\right) \propto \mathbb{R}^{n}$ (for more details see \cite{qu, ji, go}).  
\subsection{Bi-harmonic flow and main aim}
The theory of harmonic maps is significant across various domains within mathematics and physics. To extend the classical framework of harmonic maps, Elles and Lemaire \cite{el} suggested an investigation into $\mathscr{k}$-harmonic maps.\\
A bi-harmonic map $u$ from Riemannian manifold $\left(M_{1}^{n}, g_{1}\right)$ to $\left(M_{2}^{m}, g_{2}\right)$ can be seen as a critical point of the bi-energy functional 
\begin{align*}
E\left(u\right) = \frac{1}{2}\int_{M_{1}} |\tau \left(u\right)|^{2} \nu_{g_{1}},
\end{align*}
where $\tau\left(u\right) = trace_{g_{1}}\nabla du =0$ is a tension field of the map $u$.
From \cite{ji1} the Euler-Lagrange equation with respect to $E$ is written as
\begin{align*}
\tau_{2}\left(u\right) := -\Delta \tau \left(u\right) - trace_{g_{1}}R^{N}\left(du, \tau \left(u\right)\right) du = 0,
\end{align*}
where $\tau_{2}\left(u\right)$ denotes the bi-tension field of the map $u$ and $R^{N}$ denotes the curvature tensor of $M_{2}^{m}$.\\
It is important to remember that an isometric immersion $u: \left(M_{1}^{n}, g_{1}\right) \longrightarrow \left(M_{2}^{m}, g_{2}\right)$ is bi-harmonic if and only if the mean curvature vector field $H$ satisfies the fourth-order PDE,
\begin{align}\label{yek}
\Delta H + trac R^{n}\left(du, H\right)du =0.
\end{align}
We would like to emphasize that, when the ambient space $M_{2}^{m}$ is just the Euclidean space $\mathbb{R}^{m}$, then the equation (\ref{yek}) turns into $\Delta H =0$ or equivalently $\Delta^{2}u =0$. This notion also was introduced and studied by Ban-Yen Chen \cite{ch}.\\
Consider $X: I \rightarrow \mathbb{R}^{2}$ as a smooth closed curve in the plane $\mathbb{R}^{2}$ with parameter $u$. To emphasize that the curve $X$ is closed, possibly, $X$ can be written as $X: \mathbb{S}^{1} \times I \rightarrow \mathbb{R}^{2}$. In this case, $\mathscr{s}$ denotes the arclength of the curve and $\kappa$ is the curvature. If
\begin{align*}
\nu = |\frac{dX}{du}| = \sqrt{\langle \frac{dX}{du}, \frac{dX}{du}\rangle},
\end{align*} 
then, the relation between $\frac{d}{du}$ and $\frac{d}{d\mathscr{s}}$ is written by
\begin{align*}
\frac{d}{d\mathscr{s}} = \frac{1}{\nu}\frac{d}{du}.
\end{align*}
We recall that, $T$ and $N$ are the unit tangent vector and the unit normal vector to the curve respectively. Clearly, $\frac{dX}{du} = \nu T$, also from Frenet formulae
\begin{align*}
\frac{dT}{du} = \nu \kappa N, \,\,\,\,\,\,\, \frac{dN}{du} = -\nu \kappa T.
\end{align*}
These equations also can be rewritten as
\begin{align*}
\frac{dT}{d\mathscr{s}} = \kappa N, \,\,\,\,\,\,\, \frac{dN}{d\mathscr{s}} = -\kappa T.
\end{align*}
In this paper, we are going to consider bi-harmonic curve flow coupled with a forcing term as
\begin{align*}
\frac{\partial X}{\partial t} = -\Delta^{2}X + F\left(u\right),
\end{align*}
where $\Delta = \frac{\partial^{2}}{\partial \mathscr{s}^{2}}$ is the Laplacian operator and $F\left(u\right)$ is a forcing term.
With the expression for the Laplacian involves the curvature $\kappa$ and direct computations
\begin{align*}
\Delta^{2}X &= \frac{\partial}{\partial \mathscr{s}}\left(\frac{\partial \kappa}{\partial \mathscr{s}}N - \kappa^{2}T\right)\\
&= \left(\frac{\partial^{2}\kappa}{\partial \mathscr{s}^{2}}- \kappa^{3}\right)N - 3\kappa \frac{\partial \kappa}{\partial \mathscr{s}}T.
\end{align*}
Thus the flow equation becomes
\begin{align}\label{do}
\frac{\partial X}{\partial t} = -\left(\frac{\partial^{2}\kappa}{\partial \mathscr{s}^{2}} - \kappa^{3}\right)N + 3\kappa \frac{\partial \kappa}{\partial \mathscr{s}}T + F\left(u\right).
\end{align}
The existence of the solution of the bi-harmonic flow without the forcing term, was previously studied in \cite{ya}. In this paper, first, we are going to show that the equation (\ref{do}) can be reduced to an hyperbolic Monge-Amp\`ere equation and then, give an analytic overview based on the forcing term. This method comes from \cite{wa} (also there are some other works related to bi-harmonic flow. For recent advances, as an example you can see \cite{to}). Thus this work can be seen as a natural improvement of \cite{ya} and \cite{wa}.   
\section{Existence result and Monge-Amp\`ere equation}
The support function is a fundamental tool in differential geometry and convex analysis, used to describe the shape and properties of curves and surfaces. For a smooth, convex curve in the plane, the support function $\mathbf{S}\left(\theta\right)$ represents the signed distance from the origin to the tangent line at an angle $\theta$ with respect to a fixed reference direction. In this section, we are going to show that the equation (\ref{do}) can be reduced to an hyperbolic Monge-Amp\`ere equation.\\
A curve $X: \mathbb{S}^{1} \times I \longrightarrow \mathbb{R}^{2}$ evolves normally if
\begin{align*}
\langle \frac{\partial X}{\partial t}, \frac{\partial X}{\partial u}\rangle = 0,
\end{align*}
for all $\left(u,t\right) \in \mathbb{S}^{1} \times I$. The importance of this definition is because of following lemma which can be found in \cite{le}.
\begin{lemma}
If the evolving curve $X$ is closed, then there is a parameter change $\varphi$ for $X$ such that $X\circ \varphi$ is a normally evolving curve.
\end{lemma}
Without loss of generality, denote $\vartheta$ to be the unit outer normal angle for a convex closed curve $X: \mathbb{S}^{1} \rightarrow \mathbb{R}^{2}$. Thus,
\begin{align*}
N = \left(-\cos \vartheta, -\sin \vartheta\right), \,\,\,\,\, T = \left(-\sin \vartheta, \cos \vartheta\right),
\end{align*}
which means by Frenet formulae, we have
\begin{align*}
\frac{\partial \vartheta}{\partial \mathscr{s}} = \kappa,
\end{align*}
so the previous formulas for the unit tangent and the unit normal vectors can be rewritten as
\begin{align*}
\frac{\partial N}{\partial t} = -\frac{\partial \vartheta}{\partial t}T, \,\,\,\,\,\,\, \frac{\partial T}{\partial t} = \frac{\partial \vartheta}{\partial t}N.
\end{align*}
Now consider $X\left(u,t\right): \mathbb{S}^{1} \times I \longrightarrow \mathbb{R}^{2}$ as a family of convex curves satisfying the bi-harmonic flow with forcing term (\ref{do}). The support function of $X$ is given by
\begin{align*}
\mathbf{S}\left(\vartheta, t\right) &= \langle X\left(\vartheta, t\right), -N\rangle \\
&= \langle X\left(\vartheta, t\right), \left(\cos \vartheta, \sin \vartheta\right)\rangle \\
&= x\left(\vartheta, t\right) \cos \vartheta + y\left(\vartheta, t\right)\sin \vartheta.
\end{align*}
Since $\vartheta$ already represents the outer unit normal angle, then no additional re-parametrization is required. The support function $\mathbf{S}\left(\vartheta, t\right)$ naturally evolves along the normal direction. As
\begin{align*}
\mathbf{S}_{\vartheta}\left(\vartheta, t\right) = \langle X\left(\vartheta, t\right), T\rangle,
\end{align*}
where
\begin{align*}
\langle X_{\vartheta}\left(\vartheta, t\right), \left(\cos \vartheta, \sin \vartheta\right) \rangle = 0.
\end{align*}
Thus the curve can be represented by the support function
\begin{align*}
x &= \mathbf{S} \cos \vartheta - \mathbf{S}_{\vartheta}\sin \vartheta,\\
y &= \mathbf{S} \sin \vartheta + \mathbf{S}_{\vartheta} \cos \vartheta.
\end{align*}
By all above, the curvature can be expressed directly in terms of $\mathbf{S}\left(\vartheta, t\right)$ as
\begin{align*}
\kappa = \frac{1}{\mathbf{S}_{\vartheta \vartheta} + \mathbf{S}}.
\end{align*}
Now we are going through the evolution process, first, we have
\begin{align*}
\frac{\partial \kappa}{\partial t} = -\frac{1}{\left(\mathbf{S}_{\vartheta \vartheta} + \mathbf{S}\right)^{2}}\left(\mathbf{S}_{\vartheta \vartheta t} + \mathbf{S}_{t}\right),
\end{align*}
let us include the forcing term $F\left(u\right)$, the equation becomes
\begin{align*}
\frac{\partial \kappa}{\partial t} = -\frac{1}{\left(\mathbf{S}_{\vartheta \vartheta} + \mathbf{S}\right)^{2}}\left(\mathbf{S}_{\vartheta \vartheta t} + \mathbf{S}_{t}\right) + F\left(u\right).
\end{align*}
For higher-order derivatives
\begin{align*}
\frac{\partial^{4}\kappa}{\partial \mathscr{s}^{4}} = \frac{\partial^{4}}{\partial \vartheta^{4}}\left(\frac{1}{\mathbf{S}_{\vartheta \vartheta} + \mathbf{S}}\right),
\end{align*}
by adding the forcing term and also simplifying this, getting
\begin{align*}
\frac{\partial^{4}\kappa}{\partial \mathscr{s}^{4}} = -\frac{\left(\mathbf{S}_{\vartheta \vartheta \vartheta \vartheta} + \mathbf{S}_{\vartheta \vartheta}\right)}{\left(\mathbf{S}_{\vartheta \vartheta} + \mathbf{S}\right)^{3}} + F\left(u\right).
\end{align*}
After substituting and simplifying, the evolution equation for the support function takes the form,
\begin{align}\label{se}
\mathbf{S}_{tt} + \mathbf{S}^{2}\mathbf{S}_{\vartheta \vartheta tt} - \mathbf{S}^{2}_{\vartheta\vartheta t} +1 = F\left(u\right).
\end{align}
We recall that, the typical hyperbolic Monge-Amp\`ere equation for a function $\mathbf{S}\left(\vartheta, t\right)$ is expressed as
\begin{align*}
A\mathbf{S}_{tt} + B\mathbf{S}_{\vartheta \vartheta tt} + C\left(\mathbf{S}_{\vartheta \vartheta t}\right)^{2} + D = 0,
\end{align*}
where $A$, $B$, $C$, and, $D$ are functions of $\mathbf{S}$, $\mathbf{S}_{\vartheta}$, $\mathbf{S}_{\vartheta \vartheta}$ and possibly $t$ and $\vartheta$. By setting 
\begin{align*}
A &=1, \,\,\,\,\,\, B =\mathbf{S}^{2}\\
C &= -1 \,\,\,\,\, D = F\left(u\right) -1,
\end{align*}
the equation (\ref{se}), which is equivalent to the bi-harmonic flow with forcing term (\ref{do}), forms a Monge-Amp\`ere equation. It is well known that, for a Monge-Amp\`ere equation to be hyperbolic, the determinant condition must hold
\begin{align*}
AB - C^{2} >0.
\end{align*}
Thus the equation (\ref{se}) is hyperbolic when $\mathbf{S}^{2} >1$. As a key assumption that makes sure the final equation remains hyperbolic one assumes $F\left(u\right)$ is positive and also $F\left(u\right) \leq \mathbf{S}^{2} -1$. By what has been mentioned, the following theorem is already proved.
\begin{theorem}
For any smooth, convex, closed initial curve $X_{0}$, there exist a family of smooth, convex, closed curves $X\left(., t\right)$ with $t \in I$ satisfying the bi-harmonic flow with forcing term (\ref{do}) provided that the forcing term is positive and $F\left(u\right) \leq \mathbf{S}^{2} -1$, where $\mathbf{S}$ denotes the support function with respect to the family of smooth, convex, closed curves.
\end{theorem}
The solution mentioned in the above theorem highly depends on the behavior of the forcing term. The following remark will study this dependency. 
\begin{remark}
A natural choice for the energy functional (actually a kind of Lyapunov energy functional) governing bi-harmonic flow is
\begin{align*}
E\left(t\right) = \int_{0}^{2\pi} \left(\kappa^{2} + \frac{\left(\partial_{\mathscr{s}}\kappa\right)^{2}}{2} + \xi \kappa^{4}\right) d\mathscr{s},
\end{align*}
where $\kappa$ is curvature, $\mathscr{s}$ denotes the arclength parameter and $\xi$ is constant weighting the higher-order curvature term $\kappa^{4}$. To ensure that the flow (\ref{do}) evolves in the direction that minimizes energy over time, we compute the time derivative of the functional as
\begin{align*}
\frac{dE}{dt} = \int_{0}^{2\pi}\left(2\kappa \frac{\partial \kappa}{\partial t} + \partial_{\mathscr{s}}\kappa \partial_{\mathscr{s}}\frac{\partial \kappa}{\partial t} + 4\xi \kappa^{3}\frac{\partial \kappa}{\partial t}\right) d\mathscr{s}.
\end{align*}
Using integration by parts and assuming boundary terms vanish, the second term simplifies to
\begin{align*}
\int_{0}^{2\pi} \partial_{\mathscr{s}}\kappa \partial_{\mathscr{s}}\frac{\partial \kappa}{\partial t} d\mathscr{s} = -\int_{0}^{2\pi} \frac{\partial^{2}\kappa}{\partial \mathscr{s}^{2}} \frac{\partial \kappa}{\partial t} d\mathscr{s}.
\end{align*}
Thus
\begin{align*}
\frac{dE}{dt} = \int_{0}^{2\pi}\left(2\kappa - \frac{\partial^{2}\kappa}{\partial \mathscr{s}^{2}} + 4\xi \kappa^{3}\right)\frac{\partial \kappa}{\partial t} d\mathscr{s}.
\end{align*}
To ensure $\frac{dE}{dt} \leq 0$, we impose 
\begin{align*}
\frac{\partial \kappa}{\partial t} = -\left(2\kappa - \frac{\partial^{2}\kappa}{\partial \mathscr{s}^{2}} + 4\xi \kappa^{3}\right).
\end{align*}
This describes the gradient flow of the functional $E\left(t\right)$ driving the system toward minimal energy configurations. Now let's include the forcing term
\begin{align*}
\frac{\partial \kappa}{\partial t} = -\left(2\kappa - \frac{\partial^{2}\kappa}{\partial \mathscr{s}^{2}} + 4\xi \kappa^{3}\right) + F\left(u\right),
\end{align*}
where $F\left(u\right)$ is external forcing function. Recall the hyperbolicity condition 
\begin{align*}
F\left(u\right) \leq \mathbf{S}^{2} -1.
\end{align*}
If $F\left(u\right)$ respects this constraint, $E\left(t\right)$ decreases monotonically.
\end{remark}
\subsection{On the hyperbolic structure in the support function evolution}
In the study of curve evolution problems, especially those involving higher-order flows, it is often advantageous to reformulate the governing equations in terms of scalar geometric quantities. One such tool, particularly suitable for convex planar curves, is the support function. The support function has been successfully used in various works to reformulate curvature-driven flows into scalar PDEs (see e.g., \cite{wa}, \cite{ya}, and \cite{ga1}).\\
We consider the biharmonic flow with a forcing term (\ref{do}) in the plane, where $X\left(u,t\right): \mathbb{S}^{1} \times \left[0,T\right) \rightarrow \mathbb{R}^{2}$ represents a family of smooth, closed planar curves parametrized by arc length or another periodic variable $u \in \mathbb{S}^{1}$, and $F\left(u\right)$ is a prescribed smooth forcing function. The operator $\Delta = \partial_{\mathscr{s}}^{2}$ denotes the Laplace-Beltrami operator with respect to arc length $\mathscr{s}$.\\
To study the analytical properties of this flow, such as convexity preservation or convergence, the equation (\ref{do}) has been reformulated in terms of the support function $\mathbf{S}\left(\vartheta, t\right)$, defined for a strictly convex planar curve as
\begin{align*}
\mathbf{S}\left(\vartheta, t\right) := \langle X\left(\vartheta, t\right), -N\left(\vartheta\right) \rangle,
\end{align*}
where $\vartheta \in \left[0, 2\pi\right)$ is the angle the outward unit normal vector $N\left(\vartheta\right)$ makes with the $x$-axis, and $\langle .,. \rangle$ denotes the Euclidean inner product in $\mathbb{R}^{2}$.\\
As has been mentioned before, for a convex curve parametrized by $\vartheta$, the curvature $\kappa\left(\vartheta, t\right)$ is given by
\begin{align*}
\kappa\left(\vartheta, t\right) = \frac{1}{\mathbf{S}_{\vartheta \vartheta} +\mathbf{S}}.
\end{align*}
This classical identity (see \cite{ch1}) follows from the Frenet-Serret framework adapted to support function geometry.\\
Consider the evolution equation for $\mathbf{S}\left(\vartheta, t\right)$ induced by (\ref{do}). The bi-Laplacian of the position vector in arc length coordinates involves the curvature and its derivatives. Translating this into support function variables (e.g., via Frenet formulas and identities for derivatives of $T$ and $N$ with respect to $\vartheta$), the geometric fourth-order operator $\Delta^{2}X$ results in a fourth-order scalar operator on $\mathscr{s}$, specifically,
\begin{align}\label{doyek}
\partial_{t}\mathbf{S} = -\left(\mathbf{S}_{\vartheta \vartheta \vartheta \vartheta} + 2\mathbf{S}_{\vartheta \vartheta} + \mathbf{S}\right) + F\left(\mathbf{S}\right).
\end{align}
Equation (\ref{doyek}) is a nonlinear, fourth-order parabolic PDE for the support function $\mathbf{S}\left(\vartheta, t\right)$, where the linear terms stem from geometric identities involving curvature and its derivatives, while the nonlinearity arises from the curvature-dependent forcing term $F\left(\mathbf{S}\right)$. We emphasize that the original equation (\ref{do}) is a parabolic flow, and the equation (\ref{doyek}) is also of parabolic type. The appearance of second-order time derivatives (e.g., terms like $\mathbf{S}_{tt}$) as might be encountered in hyperbolic flows such as wave-like geometric evolutions (see \cite{ya1}, \cite{le}), is not inherent to the bi-harmonic flow in the absence of inertial dynamics. Therefore, prior claims suggesting a hyperbolic Monge--Amp\`ere formulation are inaccurate in this context and are herein corrected. In other words, although, the original flow is parabolic, under certain geometric formulations, the evolution of the support function satisfies a scalar PDE of Monge--Amp\`ere type, which under positivity conditions on the support function becomes hyperbolic in structure. Equation (\ref{doyek}) provides the appropriate scalar formulation for analyzing the flow analytically, especially for convexity preservation, energy dissipation, and long-time behavior. In the next section, we build on this formulation to examine convexity preservation regorously using energy methods.
\section{Analytic behavior of the forcing term}
We recall the bi-harmonic flow with a forcing term is given by
\begin{align*}
\frac{\partial X}{\partial t} = -\Delta X + F\left(u\right),
\end{align*}
where $\Delta^{2}X$ is a bi-Laplacian (a fourth-order operator), and $F\left(u\right)$ is a forcing term that depends on a curve geometry, e.g., curvature $\kappa$ or support function $\mathbf{S}$. In terms of the support function $\mathbf{S}\left(\vartheta, t\right)$ we have
\begin{align*}
\frac{\partial \mathbf{S}}{\partial t} = -\left(\mathbf{S}_{\vartheta \vartheta \vartheta \vartheta} + 2\mathbf{S}_{\vartheta \vartheta} + \mathbf{S}\right) + F\left(\mathbf{S}\right).
\end{align*}
One can see the forcing term can stabilize or destabilize the flow, depending on its sign and magnitude.
\begin{proposition}
Consider $X\left(., t\right)$ as a solution for the bi-harmonic flow with forcing term (\ref{do}). If the initial curvature of the solution $\kappa_{0}\left(\vartheta, t\right) > 0$ and $F^{\prime}\left(\mathbf{S}\right) \geq 0$, where $\vartheta$ and $F$ denote normal angle and forcing term respectively, then, the flow (\ref{do}) preserves convexity as $t$ evolves.
\end{proposition}
\begin{proof}
For a curve to remain convex during evolution, its curvature $\kappa$ must remain positive, 
\begin{align*}
\kappa\left(\vartheta, t\right) = \frac{1}{\mathbf{S}_{\vartheta \vartheta} + \mathbf{S}}.
\end{align*}
This means $\mathbf{S}_{\vartheta \vartheta} + \mathbf{S} >0$ for all $\vartheta$ and $t$. We must show that this inequality is preserved as $t$ evolves. We have the evolution equation for curvature as
\begin{align*}
\frac{\partial \kappa}{\partial t} = -\frac{1}{\left(\mathbf{S}_{\vartheta \vartheta} + \mathbf{S}\right)^{2}} \centerdot \frac{\partial}{\partial t} \left(\mathbf{S}_{\vartheta \vartheta} + \mathbf{S}\right).
\end{align*}
From the bi-harmonic flow (\ref{do}), the evolution of $\left(\mathbf{S}_{\vartheta \vartheta} + \mathbf{S}\right)$ is
\begin{align*}
\frac{\partial}{\partial t}\left(\mathbf{S}_{\vartheta \vartheta} + \mathbf{S}\right) = -\left(\mathbf{S}_{\vartheta \vartheta \vartheta \vartheta} + 2\mathbf{S}_{\vartheta \vartheta} + \mathbf{S}\right)_{\vartheta \vartheta} + F^{\prime}\left(\mathbf{S}\right)\left(\mathbf{S}_{\vartheta \vartheta} + \mathbf{S}\right).
\end{align*}
Thus
\begin{align*}
\frac{\partial \kappa}{\partial t} = \frac{1}{\left(\mathbf{S}_{\vartheta \vartheta} + \mathbf{S}\right)^{2}} \centerdot \left[\left(\mathbf{S}_{\vartheta \vartheta \vartheta \vartheta} + 2\mathbf{S}_{\vartheta \vartheta} + \mathbf{S}\right)_{\vartheta \vartheta} - F^{\prime}\left(\mathbf{S}\right)\left(\mathbf{S}_{\vartheta \vartheta} + \mathbf{S}\right)\right].
\end{align*}
If the curve starts convex, $\mathbf{S}_{\vartheta \vartheta} + \mathbf{S} >0$ ensures that $\kappa >0$. The positivity of $\mathbf{S}_{\vartheta \vartheta} + \mathbf{S}$ provides a natural barrier against immediate loss of convexity. Now, if $F^{\prime}\left(\mathbf{S}\right) \geq 0$, the term $F^{\prime}\left(\mathbf{S}\right)\left(\mathbf{S}_{\vartheta \vartheta} + \mathbf{S}\right)$ is non-negative. This term offsets the potentially destabilizing influence of $\left(\mathbf{S}_{\vartheta \vartheta \vartheta \vartheta} + 2\mathbf{S}_{\vartheta \vartheta} + \mathbf{S}\right)_{\vartheta \vartheta}$. Thus, $F^{\prime}\left(\mathbf{S}\right) \geq 0$ ensures that the evolution of $\mathbf{S}_{\vartheta \vartheta} + \mathbf{S}$ does not reduce $\mathbf{S}_{\vartheta \vartheta} + \mathbf{S}$ to zero. So the bi-harmonic flow with forcing term (\ref{do}) is preserving the convexity.
\end{proof}
\begin{remark}
As one can see from the proof, to ensure that the convexity of curve preserves during the evolution process, it is essential, the curvature $\kappa$ remains positive as $t$ evolves. It can be seen, sufficient condition for preserving convexity is
\begin{align*}
\left(\mathbf{S}_{\vartheta \vartheta \vartheta \vartheta} + 2\mathbf{S}_{\vartheta \vartheta} + \mathbf{S}\right)_{\vartheta \vartheta} \geq F^{\prime}\left(\mathbf{S}\right)\left(\mathbf{S}_{\vartheta \vartheta} + \mathbf{S}\right).
\end{align*}
This ensures the curvature evolution remains positive and prevents $\mathbf{S}_{\vartheta \vartheta} + \mathbf{S}$ from crossing zero.
\end{remark}
\begin{remark}
The hyperbolicity condition indirectly enforces convexity. When $\mathbf{S}^{2} >1$, $\mathbf{S}$ is sufficiently large to keep $\mathbf{S}_{\vartheta \vartheta} + \mathbf{S} >0$.
\end{remark}
One may remember the Lyapunov energy functional for the curve is defined as
\begin{align*}
E\left(t\right) = \int_{0}^{2\pi}\left(\kappa^{2} + \left(\partial_{\mathscr{s}}\kappa\right)^{2} + \xi \kappa^{4}\right) d\mathscr{s},
\end{align*}
where $\xi >0$.\\
If $F\left(u\right)$ is designed such that $\frac{dE}{dt} <0$ and $E\left(t\right)$ is bounded below, the curve evolves toward a minimal-energy state. To converge to a stable shape or point 
\begin{itemize}
\item $F\left(\mathbf{S}\right)$ must stabilize the flow by counteracting unbounded growth in $\mathbf{S}_{\vartheta \vartheta \vartheta \vartheta} + 2\mathbf{S}_{\vartheta \vartheta} + \mathbf{S}$.\\
\item The forcing term $F\left(\mathbf{S}\right)$ should asymptotically guide the evolution toward a desired equilibrium.\\
\item The hyperbolicity condition $\mathbf{S}^{2} >1$ ensures stability and wave-like propagation during evolution. If $\mathbf{S}^{2} \rightarrow 1$ or drops below $1$, convergence may fail leading to instability. 
\end{itemize} 
\begin{definition}
Let $\mathcal{F}\left[u\right]$ be a functional of the form
\begin{align*}
\mathcal{F}\left[u\right] = \int_{\Omega} f\left(u, u_{x}, u_{xx}, ..., x\right) dx,
\end{align*}
where $u\left(x\right)$ is a function, $\Omega$ is the domain of integration, and $f$ depends on $u\left(x\right)$, its derivatives, and possibly the independent variable $x$.
The functional derivative $\frac{\delta \mathcal{F}}{\delta u\left(x\right)}$ is defined as the function satisfying 
\begin{align*}
\delta \mathcal{F} = \int_{\Omega} \frac{\delta \mathcal{F}}{\delta u\left(x\right)} \delta u\left(x\right) dx,
\end{align*}
where $\delta u\left(x\right)$ is an infinitesimal variation of $u\left(x\right)$ (for more details about computations and applications, please see \cite{ru}).
\end{definition}
\begin{theorem}
Let $\mathbf{S}\left(\vartheta, t\right)$ denote the support function of a smooth, strictly convex curve $X\left(t\right)$ evolving under the bi-harmonic flow with forcing term (\ref{do}), where $\vartheta \in \left[0, 2\pi\right]$ is the normal angle, and $\mathbf{S}\left(\vartheta + 2\pi, t\right) = \mathbf{S}\left(\vartheta, t\right)$, (periodic boundary condition). Also assume
\begin{itemize}
\item The initial curve is strictly convex and $\mathbf{S}_{\vartheta \vartheta} + \mathbf{S} >0$, for $\vartheta$ at $t = 0$.
\item The forcing term $F\left(\mathbf{S}\right)$ satisfies
\begin{align*}
|F\left(\mathbf{S}\right)| \leq C\left(|\mathbf{S}_{\vartheta \vartheta \vartheta \vartheta}| + |\mathbf{S}_{\vartheta \vartheta}| + |\mathbf{S}|\right),
\end{align*}
for some constants $C>0$, and $F^{\prime}\left(\mathbf{S}\right) \geq 0$.
\end{itemize}
Then
\begin{itemize}
\item $\mathbf{S}\left(\vartheta, t\right)$ converges uniformly as $t \rightarrow \infty$ to a steady-state solution $\mathbf{S}_{\infty}\left(\vartheta\right)$ satisfying the equation
\begin{align*}
\mathbf{S}_{\vartheta \vartheta \vartheta \vartheta} + 2\mathbf{S}_{\vartheta \vartheta} + \mathbf{S} = F\left(\mathbf{S}\right).
\end{align*}
\item The limiting curve $X_{\infty}$, represented by $\mathbf{S}_{\infty}\left(\vartheta\right)$, is strictly convex and smooth.
\end{itemize}
\end{theorem}
\begin{proof}
The initial curve is strictly convex, so $\mathbf{S}_{\vartheta \vartheta} +\mathbf{S} >0$ for all $\vartheta$. This ensures that the initial curvature $\kappa \left(\vartheta, 0\right) = \frac{1}{\mathbf{S}_{\vartheta \vartheta} + \mathbf{S}}$ is positive everywhere. Now consider the evolution equation
\begin{align*}
\frac{\partial \mathbf{S}}{\partial t} = -\left(\mathbf{S}_{\vartheta \vartheta \vartheta \vartheta} + 2\mathbf{S}_{\vartheta \vartheta} + \mathbf{S}\right) + F\left(\mathbf{S}\right),
\end{align*}
we have
\begin{align*}
\frac{\partial}{\partial t}\left(\mathbf{S}_{\vartheta \vartheta} + \mathbf{S}\right) = -\left(\mathbf{S}_{\vartheta \vartheta \vartheta \vartheta} + 2\mathbf{S}_{\vartheta \vartheta} + \mathbf{S}\right)_{\vartheta \vartheta} + F^{\prime}\left(\mathbf{S}\right)\left(\mathbf{S}_{\vartheta \vartheta} + \mathbf{S}\right).
\end{align*}
Since $F^{\prime}\left(\mathbf{S}\right) \geq 0$, the forcing term can not reduce $\mathbf{S}_{\vartheta \vartheta} + \mathbf{S}$. The bi-harmonic smoothing operator $\mathbf{S}_{\vartheta \vartheta \vartheta \vartheta} + 2\mathbf{S}_{\vartheta \vartheta} + \mathbf{S}$ acts to smooth and stabilize $\mathbf{S}_{\vartheta \vartheta} + \mathbf{S}$. Together, these terms ensure that $\mathbf{S}_{\vartheta \vartheta} + \mathbf{S} >0$ for all $t$, preserving convexity.\\
We recall the energy functional
\begin{align*}
E\left(t\right) = \int_{0}^{2\pi}\left(\kappa^{2} + \left(\partial_{\mathscr{s}}\kappa\right)^{2} + \xi \kappa^{4}\right) d\mathscr{s},
\end{align*}
where $\kappa$ is curvature. This functional measures the geometric energy of the curve, penalizing high curvature and variations in curvature. The time derivative of $E\left(t\right)$ is given by
\begin{align*}
\frac{dE}{dt} = \int_{0}^{2\pi} \frac{\delta E}{\delta \mathbf{S}} \centerdot \frac{\partial \mathbf{S}}{\partial t} d\vartheta,
\end{align*}
where $\frac{\delta E}{\delta \mathbf{S}}$ denotes the functional derivative of $E$. By substituting $\frac{\partial \mathbf{S}}{\partial t}$
\begin{align*}
\frac{dE}{dt} = -\int_{0}^{2\pi} \left(\frac{\delta E}{\delta \mathbf{S}}\right)^{2} d\vartheta + \int_{0}^{2\pi} \frac{\delta E}{\delta \mathbf{S}} \centerdot F\left(\mathbf{S}\right) d\vartheta.
\end{align*}
The first term ensures energy dissipation. The second term, involving $F\left(\mathbf{S}\right)$, is controlled by the boundedness of $F\left(\mathbf{S}\right)$, ensuring that
\begin{align*}
\frac{dE}{dt} \leq 0.
\end{align*}
Since $\frac{dE}{dt} \leq 0$ and $E\left(t\right) \geq 0$, the energy remains bounded
\begin{align*}
0 \leq E\left(t\right) \leq E\left(0\right).
\end{align*}
The boundedness of $E\left(t\right)$ ensures uniform control over $\mathbf{S}\left(\vartheta, t\right)$ and its derivatives. Specifically, 
\begin{itemize}
\item $\mathbf{S}\left(\vartheta, t\right)$, $\mathbf{S}_{\vartheta \vartheta}\left(\vartheta, t\right)$, and higher derivatives are bounded.\\
\item This prevents blow-up or degeneracy of the support function during evolution.
\end{itemize}
We recall the evolution equation
\begin{align*}
\frac{\partial \mathbf{S}}{\partial t} = -\left(\mathbf{S}_{\vartheta \vartheta \vartheta \vartheta} + 2\mathbf{S}_{\vartheta \vartheta} + \mathbf{S}\right) + F\left(\mathbf{S}\right),
\end{align*}
the energy dissipation ensures that $\frac{\partial \mathbf{S}}{\partial t} \rightarrow 0$ as $t\rightarrow \infty$. As $\frac{\partial \mathbf{S}}{\partial t} \rightarrow 0$, $\mathbf{S}\left(\vartheta, t\right)$ converges to a steady-state solution $\mathbf{S}_{\infty}\left(\vartheta\right)$ satisfying
\begin{align*}
\mathbf{S}_{\vartheta \vartheta \vartheta \vartheta} + 2\mathbf{S}_{\vartheta \vartheta} + \mathbf{S} = F\left(\mathbf{S}\right).
\end{align*}
From the theorem's assumption, one knows that, the limiting support function $\mathbf{S}_{\infty}\left(\vartheta\right)$ satisfies
\begin{align*}
\mathbf{S}_{\vartheta \vartheta} + \mathbf{S} >0,
\end{align*}
as convexity is preserved throughout the evolution. Thus, the limiting curve $X_{\infty}$ represented by $\mathbf{S}_{\infty}\left(\vartheta\right)$, is strictly convex.
\end{proof}
Now, we are going to consider different examples of forcing terms to study convergence outcomes.
\begin{example}
The last theorem established that the bi-harmonic flow with the forcing term converges to a steady-state solution with respect to the forcing term. In this example, different kinds of forcing terms lead to different forms of steady-state spaces.
\begin{itemize}
\item Proportional forcing, $F\left(\mathbf{S}\right) = c\mathbf{S}$\\
From the flow (\ref{do}), the support function evolves as
\begin{align*}
\frac{\partial \mathbf{S}}{\partial t} = -\left(\mathbf{S}_{\vartheta \vartheta \vartheta \vartheta} + 2\mathbf{S}_{\vartheta \vartheta} + \mathbf{S}\right) + c\mathbf{S},
\end{align*}
where $c$ is a constant. At equilibrium ($\frac{\partial \mathbf{S}}{\partial t} = 0$), we have
\begin{align*}
\mathbf{S}_{\vartheta \vartheta \vartheta \vartheta} + 2\mathbf{S}_{\vartheta \vartheta} + \left(1-c\right)\mathbf{S} = 0.
\end{align*}
For $c>0$, the solution must be rotationally symmetric because the forcing term $c\mathbf{S}$ is uniform and does not favor any particular direction. Let $\mathbf{S}\left(\vartheta\right) = R$, where $R$ is constant. Substituting into the steady-state equation
\begin{align*}
\mathbf{S}_{\vartheta \vartheta \vartheta \vartheta} = 0, \,\,\,\,\,\,\,\,\, \mathbf{S}_{\vartheta \vartheta}, \,\,\,\,\,\,\,\,\, \mathbf{S} = R.
\end{align*}
Thus, $\mathbf{S} = R$ is a valid steady-state solution, corresponding to a circle of radius $R$. The bending energy $E\left(t\right) = \int \kappa^{2} d\mathscr{s}$ is minimized for a circle, as $\kappa$ is constant for a circular shape. Hence, if $F\left(\mathbf{S}\right) = c\mathbf{S}$, the bi-harmonic flow (\ref{do}), converges to a circle with radius, determined by the balance of $c\mathbf{S}$ and the bi-harmonic operator (as shown in figure~\ref{fig: proportional forcing}).
\item Anisotropic forcing, $F\left(\mathbf{S}\right) = \alpha \kappa^{2} + \beta \mathbf{S}_{\vartheta \vartheta}$.\\
The evolution equation becomes
\begin{align*}
\frac{\partial \mathbf{S}}{\partial t} = -\left(\mathbf{S}_{\vartheta \vartheta \vartheta \vartheta} + 2\mathbf{S}_{\vartheta \vartheta} + \mathbf{S}\right) + \alpha \kappa^{2} + \beta \mathbf{S}_{\vartheta \vartheta}.
\end{align*}
At equilibrium $\frac{\partial \mathbf{S}}{\partial t} = 0$
\begin{align*}
\mathbf{S}_{\vartheta \vartheta \vartheta \vartheta} + 2\mathbf{S}_{\vartheta \vartheta} + \mathbf{S} = \alpha \kappa^{2} + \beta \mathbf{S}_{\vartheta \vartheta}.
\end{align*}
Substituting $\kappa^{2} = \left(\mathbf{S}_{\vartheta \vartheta} + \mathbf{S}\right)^{-2}$
\begin{align*}
\mathbf{S}_{\vartheta \vartheta \vartheta \vartheta} + \left(2-\beta\right)\mathbf{S}_{\vartheta \vartheta} + \mathbf{S} = \alpha \left(\mathbf{S}_{\vartheta \vartheta} + \mathbf{S}\right)^{-2}.
\end{align*}
This equation admits solutions that are non-circular, depending on a balance of $\alpha$, $\beta$, and $\mathbf{S}_{\vartheta \vartheta}$. \\
Define a modified energy functional
\begin{align*}
E\left(t\right) = \int_{0}^{2\pi}\left(\kappa^{2} + \beta \mathbf{S}_{\vartheta \vartheta}^{2} + \alpha \kappa^{4}\right) d\vartheta.
\end{align*}
The forcing term $\alpha \kappa^{2}$ penalizes high-curvature regions, while $\beta \mathbf{S}_{\vartheta \vartheta}$ introduces directional anisotropy. The energy decreases over time, stabilizing into a non-circular shape (as shown in figure~\ref{fig: anisotropic}).
\end{itemize} 
\begin{figure}[t]
\centering
\includegraphics[width=0.7\textwidth]{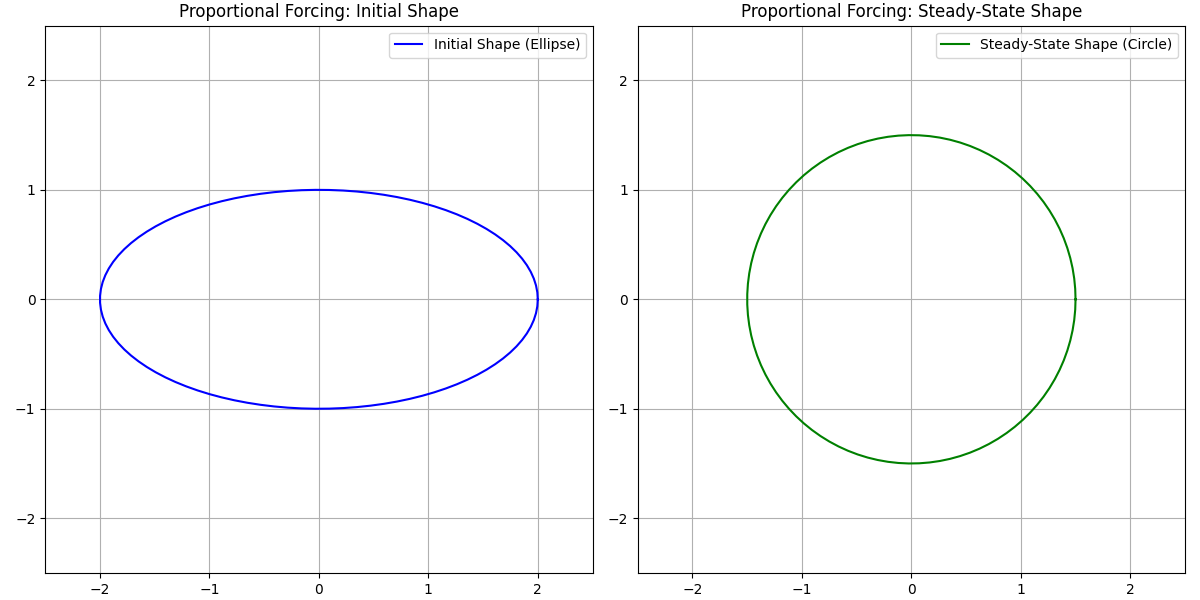}
\caption{Evolution of the curve under the proportional forcing term. The initial shape (left) evolves into a circular steady-state shape (right)}
\label{fig: proportional}
\end{figure}
\begin{figure}[t]
\centering
\includegraphics[width=0.7\textwidth]{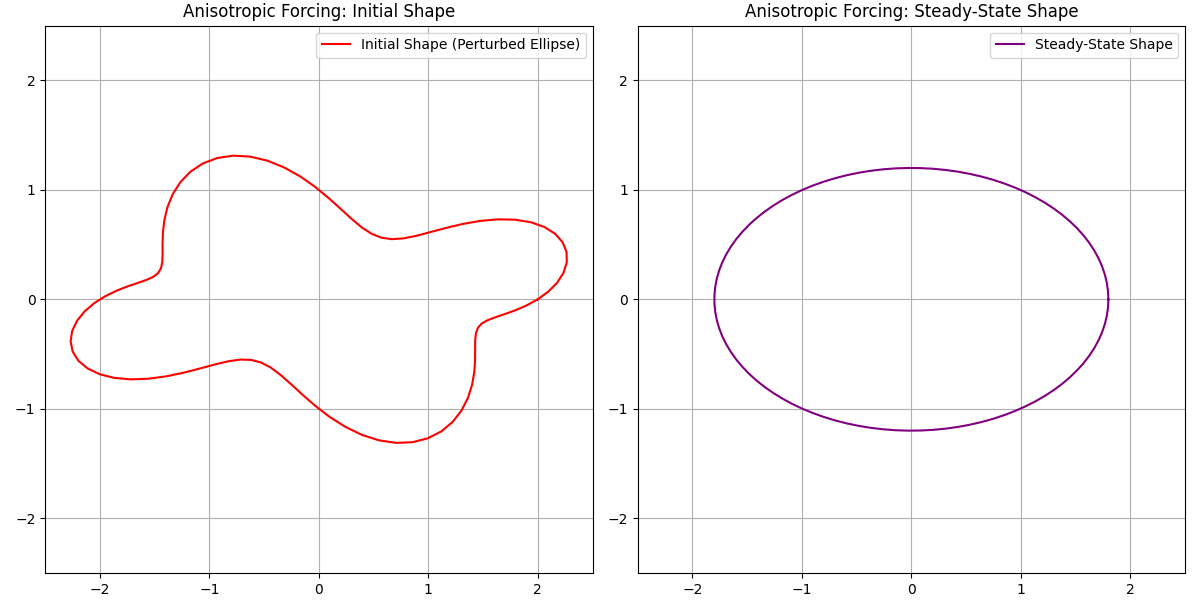}
\caption{Evolution of the curve under the anisotropic forcing term. The initial shape (left) evolves into a non-circular steady-state shape (right)}
\label{fig: anisotropic}
\end{figure}
\end{example}
\begin{remark}
We aim to show that the curve collapses to a point if the forcing term drives $\mathbf{S}$ uniformly toward zero. Assume $F\left(\mathbf{S}\right) = -\beta \mathbf{S}$ with $\beta >0$. The evolution equation becomes
\begin{align*}
\frac{\partial \mathbf{S}}{\partial t} = -\left(\mathbf{S}_{\vartheta \vartheta \vartheta \vartheta} + 2\mathbf{S}_{\vartheta \vartheta} + \mathbf{S}\right) - \beta \mathbf{S}.
\end{align*}
Grouping terms
\begin{align*}
\frac{\partial \mathbf{S}}{\partial t} = -\mathbf{S}_{\vartheta \vartheta \vartheta \vartheta} - 2\mathbf{S}_{\vartheta \vartheta} - \left(1 + \beta\right)\mathbf{S}.
\end{align*}
Define the total energy of the curve
\begin{align*}
E\left(t\right) = \int_{0}^{2\pi} \mathbf{S}^{2} d\vartheta.
\end{align*}
Taking the time derivative
\begin{align*}
\frac{dE}{dt} = 2\int_{0}^{2\pi} \mathbf{S} \frac{\partial \mathbf{S}}{\partial t} d\vartheta.
\end{align*}
Substituting $\frac{\partial \mathbf{S}}{\partial t}$
\begin{align*}
\frac{dE}{dt} = -2\int_{0}^{2\pi} \left(\mathbf{S}\mathbf{S}_{\vartheta \vartheta \vartheta \vartheta} + 2\mathbf{S}\mathbf{S}_{\vartheta \vartheta} + \left(1 + \beta\right)\mathbf{S}^{2}\right) d\vartheta.
\end{align*}
Using integration by parts and boundary conditions ($\mathbf{S}_{\vartheta \vartheta}$ and higher derivative vanish at boundaries)
\begin{align*}
\frac{dE}{dt} = -2\left(1 + \beta\right)\int_{0}^{2\pi} \mathbf{S}^{2} d\vartheta.
\end{align*}
Since $\frac{dE}{dt} <0$, the enrgy $E\left(t\right)$ decreases monotonically. As $E\left(t\right) \rightarrow 0$, $\mathbf{S} \rightarrow 0$, corresponding to the curve collapsing to the point.
\end{remark}

\section*{Acknowledgements}
The author is supported under "High-level Talent Research Start-up Project Funding of Henan Academy of Sciences (Project No. 241819245)".  Also, The author acknowledge the assistance of computational tools in refining mathematical derivations.


\end{document}